\documentclass[12pt, reqno]{amsart}
\usepackage{amsmath,amsfonts,amsbsy,amsgen,amscd,mathrsfs,amssymb,amsthm}
\usepackage[usenames,dvipsnames]{xcolor}
\usepackage[colorlinks=true,citecolor=red,linkcolor=blue]{hyperref}
\usepackage[a4paper,margin=1in]{geometry}
\usepackage{setspace}
\usepackage{amsmath}
\allowdisplaybreaks[4]
\numberwithin{equation}{section}
\newtheorem{theorem}{Theorem}[section]

\newtheorem{lemma}[theorem]{Lemma}

\newtheorem{remark}[theorem]{Remark}

\title[Gauss curvature flow to the $L_p$-Gaussian chord Minkowski problem]{Gauss curvature flow to the $L_p$-Gaussian chord Minkowski problem}

\author{Xia Zhao and Peibiao Zhao}


\begin{document}
\begin{abstract}
Recently, Huang and Qin \cite{HY01} introduced the Gaussian chord measure and $L_p$-Gaussian chord measure by variational methods. Meanwhile, they posed Gaussian chord Minkowski problem for $p=1$ and used variational methods to obtain an origin-symmetric normalized measure solution for the Gaussian chord Minkowski problem. The smooth solution, up to now, to the $L_p$-Gaussian chord Minkowski problem is still open.

Motivated by the forgoing works by Huang and Qin in \cite{HY01}, we propose in the present paper the $L_p(p>0)$-Gaussian chord Minkowski problem and log-Gaussian chord Minkowski problem, and obtain the smooth even solutions to these two types of problems by the method of a Gauss curvature flow.
\end{abstract}

\maketitle

\vskip 20pt
\section{Introduction and main results}

In the past 30 years, the Minkowski problem has played an important role in the study of convex geometry, and the research of Minkowski problem has promoted the development of fully nonlinear partial differential equations. The classical Minkowski problem argues the existence, uniqueness and regularity of a convex body whose surface area measure is equal to a pre-given Borel measure on the sphere. If the given measure has a positive continuous density, the Minkowski problem can be seen as the problem of prescribing the Gauss curvature in differential geometry. With the emergence of $L_p$ surface measure, the $L_p$ Minkowski problem has also been proposed. In particularly, the $L_p$($p\in \mathbb{R}$)-Minkowski problem is extremely important, because it contains some special versions, such as when $p=1$, it is the classical Minkowski problem; when $p=0$, the critical $L_0$ case which is called cone volume measure, the Minkowski problem for cone volume measure is the log-Minkowski problem \cite{BO}; when $p=-n$, it corresponds to the centro-affine Minkowski problem \cite{ZG}; the $L_p$-Minkowski problem with $p>1$ was first proposed and studied by Lutwak \cite{LE0}, whose solution plays a key role in establishing the $L_p$-affine Sobolev inequality \cite{HC2, LE01}.

With the development of Minkowski problems, scholars have proposed some types of Minkowski problems. We know that the different geometric measures are corresponding to different Minkowski type problems. For instance, Xiao \cite{XJ} prescribed capacitary curvature measures on planar convex domains: If a given finite, nonnegative Borel measure $\mu \in S^1$ has centroid at the origin and its supp($\mu$) does not comprise any pair of antipodal points, then, there is a unique (up to translation) convex, nonempty, open set $\Omega\subset \mathbb{R}^2$ such that $d\mu_q(\Omega,\cdot)=d\mu(\cdot)$, where $\mu_q(\Omega,\cdot)$ is $q$-capacitary curvature measure of $\Omega$ with $q\in(1,2]$. 

Recently, a theory analogous to the one for the Minkowski problem was introduced by Lutwak, Xi, Yang, and Zhang \cite{LE1}, the volume is replaced by ($q$) chord integrals and surface area measure is replaced by the differential of ($q\geq0$) chord integrals which is defined chord measure. Based on chord measure, they posed and solved chord Minkowski problem and chord log-Minkowski problem (partly solved). Interestingly, when $q=1$, the chord measure recovers surface area measure, and when $(q=0)$, it is the area measure $S_{n-2}$. For the $L_p$ chord Minkowski problem, Xi, Yang, Zhang and Zhao \cite{XD} used variational methods gave a measure solution when $p>1$ and $0<p<1$ in the symmetric case. Guo, Xi and Zhao \cite{GL} also obtained a measure solution for $0\leq p<1$ by similar methods without the symmetric assumption. In addition, authors \cite{ZX} posed the Orlicz chord Minkowski problem and obtained smooth solutions by a Gauss curvature flow. When $q=1$, the $L_p$ chord Minkowski problem is the $L_p$ Minkowski problem. For other references with respect to the $L_p$ chord Minkowski problem, please refer to \cite{HJ1, HJ2, LYY, QL}.

Huang, Xi and Zhao \cite{HY02} extended the classical Minkowski problem in $\mathbb{R}^n$ to the Gaussian probability space, where the volume and surface area measure in Euclidean space were replaced by Gaussian volume and Gaussian surface measure, respectively. Then the Gaussian Minkowski problem was posed and provided a sufficient condition for the existence of solution. Later, Liu \cite{LJ} extended the Gaussian Minkowski problem to $L_p$ form, when $p\geq1$, a sufficient condition for the existence and uniqueness of origin-symmetric weak solution was given. Moreover, Feng, Hu and Xu \cite{FE} provided the existence of symmetric (resp. asymmetric) solutions to the problem for $p\leq0$ (resp. $p\geq 1$). Furthermore, Sheng and Xue \cite{SW} obtained smooth solutions by method of Gauss curvature flow for $p>0$ and $-n-1<p\leq 0 $ to normalized $L_p$ Gaussian Minkowski problem and $p\geq n+1$ and $0<p<n+1$ to $L_p$ Gaussian Minkowski problem. Contrary to classical surface area measure, the Gaussian surface measure is neither translationally invariant nor homogeneous. These special properties make the Gaussian Minkowski problem quite differ from the classical Minkowski problem, which has an interest on its own.

Very recently, based on chord integral and  Gaussian probability space, Huang and Qin \cite{HY01} generalized the chord integrals to Gaussian probability space for $q>1$. In the case of $q>1$, the Gaussian chord integral is defined by

\begin{align}\label{eq101} 
\nonumber I_{\gamma,q}(K)&=\int_K\int_K\frac{1}{|z-y|^{n-q+1}}d\mathcal{H}^n_{\gamma}(z)d\mathcal{H}^n_{\gamma}(y)\\
&=\int_K\int_K\frac{e^{-(|z|^2+|y|^2)/2}}{|z-y|^{n-q+1}}dzdy.
\end{align}

Furthermore, they established ($L_p$) the variational formula of Gaussian chord integral with ($L_p$) Minkowski sum. Here, we first state the variational formula for $p=1$. 

\begin{theorem}\label{the10}\cite[Theorem 4.3]{HY01}~~~~Let $K\in\mathcal{K}_o^n$ and $q>1$. Suppose that $g:S^{n-1}\rightarrow \mathbb{R}$ is a continuous function and $h_t:S^{n-1}\rightarrow \mathbb{R}$ is given by
\begin{align*}
h_t(u)=h_K(u)+tg(u)+o(t,u), \quad t\in (-\delta,+\delta) \quad \text{and}\quad u\in S^{n-1},
\end{align*}
where $o(t,u)/t\rightarrow 0$ uniformly on $S^{n-1}$, as $t\rightarrow 0$. If
\begin{align*}
K_t=\{x\in \mathbb{R}^n:x\cdot u\leq h_t(u)~~\text{for all}~~u\in S^{n-1}\}, \quad t\in (-\delta,+\delta),
\end{align*}
is the Wulff shape of $h_t$, it holds 
\begin{align*}
\lim_{t\rightarrow 0}\frac{\rho_{K_t}(u)-\rho_K(u)}{t}=\frac{g(\nu_K(\rho_K(u)u))}{u\cdot\nu_K(\rho_K(u)u)}.
\end{align*}
Then
\begin{align*}
\frac{d}{dt}\bigg| I_{\gamma,q}(K_t)=&\lim_{t\rightarrow 0}\frac{I_{\gamma,q(K_t)}-I_{\gamma,q(K)}}{t}\\
=&\lim_{t\rightarrow 0}\frac{1}{t}\bigg\{\int_{K}\int_{K}\frac{1}{|z-y|^{n-q+1}d\gamma(z)d\gamma(y)}-\int_{K_t}\int_{K_t}\frac{1}{|z-y|^{n-q+1}d\gamma(z)d\gamma(y)}\bigg\}\\
=&\int_{S^{n-1}}\int_{S^{n-1}}\lim_{t\rightarrow 0}\bigg\{\int_0^{\rho_{K_t}(u)}\int_0^{\rho_{K_t}(v)}-\int_0^{\rho_{K}(u)}\int_0^{\rho_{K}(v)}\bigg\}
\frac{r_1^{n-1}r_2^{n-1}e^{-(r_1^2+r_2^2)/2}}{|r_1u-r_2v|^{n-q+1}}dr_1dr_2dudv\\
=&2\int_{S^{n-1}}\int_{S^{n-1}}\int_0^{\rho_K(v)}\frac{\rho_K(u)^{n-1}r_2^{n-1}e^{-(\rho_K(u)^2+r_2^2)/2}}{|\rho_K(u)u-r_2v|^{n-q+1}}\lim_{t\rightarrow 0}\frac{\rho_{K_t}(u)-\rho_K(u)}{t}dr_2dudv\\
=&2\int_{S^{n-1}}\int_{S^{n-1}}\int_0^{\rho_K(v)}\frac{\rho_K(u)^{n-1}r_2^{n-1}e^{-(\rho_K(u)^2+r_2^2)/2}}{|\rho_K(u)u-r_2v|^{n-q+1}}
\frac{g(\nu_K(\rho_K(u)u))}{u\cdot\nu_K(\rho_K(u)u)}dr_2dudv\\
=&2\int_{\partial K}\int_K\frac{e^{-z^2/2}}{|y-z|^{n-q+1}}dz\cdot g(\nu_K(y))e^{-y^2/2}d\mathcal{H}^{n-1}(y)\\
=&\int_{\partial K}g(\nu_K(y))\widetilde{V}_{\gamma,q}(K,y)d\mathcal{H}^{n-1}(y)=\int_{S^{n-1}}g(u)dF_{\gamma,q}(K,u),
\end{align*}
where
\begin{align}\label{eq102} 
\widetilde{V}_{\gamma,q}(K,y)=2e^{-|y|^2/2}\bigg\{\int_K\frac{e^{-|z|^2/2}}{|z-y|^{n-q+1}}dz\bigg\},
\end{align}
and
\begin{align*}
{F}_{\gamma,q}(K,\eta)=\int_{\nu_K^{-1}(\eta)}\widetilde{V}_{\gamma,q}(K,y)d\mathcal{H}^{n-1}(y), \quad \text{Borel set}~~\eta\subseteq S^{n-1}.
\end{align*}
\end{theorem}
Next, a similar result holds for the $L_p$-combination perturbation of the supporting function, which is stated in the following theorem.
\begin{theorem}\label{the11}\cite[Theorem 4.4]{HY01}~~~~Let $p\neq 0$, $q>1$ and $K\in\mathcal{K}^n_o$. Suppose that $g:S^{n-1}\rightarrow \mathbb{R}$ is a continuous function and $h_t:S^{n-1}\rightarrow \mathbb{R}$ is given by
\begin{align*}
h_t(u)=(h_K(u)^p+tg(u)^p)^{\frac{1}{p}}, \quad u\in S^{n-1}
\end{align*}
and $K_t$ is the Wulff shape of $h_t$, it holds
\begin{align*}
\lim_{t\rightarrow 0}\frac{\rho_{K_t}(u)-\rho_K(u)}{t}=\frac{1}{p}\frac{g(\nu_K(\rho_K(u)u))^ph_K(\nu_K(\rho_K(u)u))^{1-p}}{u\cdot\nu_K(\rho_K(u)u)}.
\end{align*}
Then
\begin{align}\label{eq103} 
\nonumber\frac{d}{dt}\bigg|_{t=0}I_{\gamma,q}(K_t)=&\int_{\partial K}g(\nu_K(y))^p\widetilde{V}_{\gamma,q}(K,y)h_K(\nu_K(y))^{1-p}d\mathcal{H}^{n-1}(y)\\
=&\int_{S^{n-1}}g(u)^pdG_{\gamma,p,q}(K,u),
\end{align}
where
\begin{align}\label{eq104} 
G_\gamma^{p,q}=G_{\gamma,p,q}(K,\eta)=\frac{2}{p}\int_{\nu_K^{-1}(\eta)}\widetilde{V}_{\gamma,q}(K,y)h_K(\nu_K(y))^{1-p}d\mathcal{H}^{n-1}(y), \text{Borel set} ~~\eta\subseteq S^{n-1}.
\end{align}
\end{theorem}
In fact, we can directly obtain Theorem \ref{the11} by replacing $\lim_{t\rightarrow0}\frac{\rho_{K_t}(u)-\rho_K(u)}{t}$ with $\frac{1}{p}\frac{g(\nu_K(\rho_K(u)u))^ph_K(\nu_K(\rho_K(u)u))^{1-p}}{u\cdot\nu_K(\rho_K(u)u)}$ in Theorem \ref{the10}.
When $p=1$, (\ref{eq103}) correspondings to the Gaussian chord measure in Theorem \ref{the10}.

As for the $\log$-Minkowski perturbation, the variational formula of nonlocal function in Gaussian probability space is obtained in the following Theorem.

\begin{theorem}\label{the12}\cite[Theorem 4.5]{HY01}~~~~Let $q>1$ and $K\in\mathcal{K}^n_o$. Suppose that $g:S^{n-1}\rightarrow \mathbb{R}$ is a continuous function and $h_t:S^{n-1}\rightarrow \mathbb{R}$ is given by
\begin{align*}
\log h_t(u)=&\log h_K(u)+tg(u)+o(t,u)\quad t\in (-\delta,+\delta)\quad \text{and} \quad u\in S^{n-1},
\end{align*}
where, $o(t,\cdot)/t \rightarrow 0$ uniformly on $S^{n-1}$, as $t\rightarrow 0$. If $K_t$ is the Wulff shape of $h_t$, it holds
\begin{align*}
\lim_{t\rightarrow0}\frac{\rho_{K_t}(u)-\rho_K(u)}{t}=&\rho_{K_t}^\prime(u)|_{t=0}=\rho_{K}(u)\log \rho_{K_t}^\prime(u)|_{t=0}\\
=&\rho_{K}(u)\lim_{t\rightarrow 0}\frac{\log\rho_{K_t}(u)-\log\rho_K(u)}{t}=h_K(\nu_K(\rho_K(u)u))\frac{g(\nu_K(\rho_K(u)u))}{u\cdot\nu_K(\rho_K(u)u)}.
\end{align*}
Then
\begin{align}\label{eq105} 
\nonumber\frac{d}{dt}\bigg|_{t=0}I_{\gamma,q}(K_t)=&\int_{\partial K}g(\nu_K(y))\widetilde{V}_{\gamma,q}(K,y)h_K(\nu_K(y))d\mathcal{H}^{n-1}(y)\\
=&\int_{S^{n-1}}g(u)dG_{\gamma,0,q}(K,u),
\end{align}
where
\begin{align}\label{eq106} 
G_\gamma^{\log,q}=G_{\gamma,0,q}(K,\eta)=\int_{\nu_K^{-1}(\eta)}\widetilde{V}_{\gamma,q}(K,y)h_K(\nu_K(y))d\mathcal{H}^{n-1}(y), \text{Borel set} ~~\eta\subseteq S^{n-1}.
\end{align}
\end{theorem}

In the same way, we immediately get Theorem \ref{the12} by replacing $\lim_{t\rightarrow0}\frac{\rho_{K_t}(u)-\rho_K(u)}{t}$ with $h_K(\nu_K(\rho_K(u)u))\frac{g(\nu_K(\rho_K(u)u))}{u\cdot\nu_K(\rho_K(u)u)}$ in Theorem \ref{the10}.

With the help of variational formula (\ref{eq103}), we can propose the  normalized $L_p$-Gaussian chord Minkowski problem for $p \neq 0$.

{\bf The $L_p$-Gaussian chord Minkowski problem.} Let $q>1$, $p\neq 0$ and $\mu$ be a finite Borel measure on $S^{n-1}$, under what necessary and sufficient conditions, does there exist a unique convex body $\Omega\in\mathcal{K}_o^n$ and positive constant $\tau_1$ so that
\begin{align}\label{eq107}
\mu=\tau_1G_{\gamma}^{p,q}(\Omega,\cdot)?
\end{align}

From \cite[Proposition 3.7]{HY01}, we know that $G_{\gamma}^{p,q}(\Omega,\cdot)$ is absolutely continuous with respect to the surface measure $S(K,\cdot)$. If the given measure $\mu$ is absolutely continuous with respect to the spherical Lebesgue measure, that is, $µ$ has a density function $f:S^{n-1}\rightarrow (0,\infty)$ is smooth, then, solving problem (\ref{eq107}) can be equivalently viewed as solving the following normalized Monge-Amp\`{e}re equation from the (\ref{eq104}) for $p\neq0$ on $S^{n-1}$.
\begin{align}\label{eq108} 
\tau_1 \widetilde{V}_{\gamma,q}(\Omega,\cdot)h_{\Omega}^{1-p}\det(h_{ij}+h\delta_{ij})=f.
\end{align}

When $p=1$, the $L_p$-Gaussian chord Minkowski problem is the Gaussian chord Minkowski problem which was first studied by Huang and Qin \cite{HY01}, they obtained an origin-symmetric normalized measure solution by variational method.

Similar to cone volume measure, we call $G_\gamma^{\log,q}$ cone-Gaussian chord measure. The Minkowski problem prescribing cone-Gaussian chord measures is:

{\bf The log-Gaussian chord Minkowski problem.} Let $q>1$ and $\mu$ be a finite Borel measure on $S^{n-1}$, under what necessary and sufficient conditions, does there exist a unique convex body $\Omega\in\mathcal{K}_o^n$ and positive constant $\tau_2$ so that
\begin{align}\label{eq107+}
\mu=\tau_2G_{\gamma}^{\log,q}(\Omega,\cdot)?
\end{align}

The partial differential equation associated with (\ref{eq107+}) is a new type of Monge-Amp\`{e}re
equation from (\ref{eq106}) on $S^{n-1}$,
\begin{align}\label{eq108+} 
\tau_2 \widetilde{V}_{\gamma,q}(\Omega,\cdot)h_{\Omega}\det(h_{ij}+h\delta_{ij})=f.
\end{align}

In this paper, we will study the $L_p$-Gaussian chord Minkowski problem and give the existence of smooth even solutions for (\ref{eq108}) with $q>2$ and $p>0$ by the method of a Gauss curvature flow. In addition, we also provide a smooth even solution to log-Gaussian chord Minkowski problem to (\ref{eq108+}). The Gauss curvature flow was first introduced and studied by Firey \cite{FI} to model the shape change of worn stones. It can mainly be used to study the existence of smooth solutions to the famous Minkowski (type) problems. For instances, Chen, Huang and Zhao \cite{CC} obtained smooth even solutions to the $L_p$ dual Minkowski problem by the method of Gauss curvature flow. Liu and Lu \cite{LY} used a Gauss curvature flow to solve dual Orlicz-Minkowski problem and obtained its smooth solutions. Since then, various problems of Gauss curvature flows have been extensively studied, see examples \cite{AB1, BR, BR1, CK0, LR, LY1} and the references therein.

\begin{remark} From the definitions of the $L_p(p\neq 0)$-combination perturbation and log-Minkowski perturbation of the support function, we konw that the log-Minkowski perturbation is the limit form of the $L_p$ combination perturbation when $p$ tends to $0$. In this sense, the variational formula (\ref{eq103}) and (\ref{eq105}) belong to different categories, then, the corresponding $L_p(p\neq 0)$-Gaussian chord Minkowski problem (\ref{eq107}) is different from log-Gaussian chord Minkowski problem (\ref{eq107+}). However, when we transform (\ref{eq107}) to equivalent form equation (\ref{eq108}) and (\ref{eq107+}) to (\ref{eq108+}), from the perspective of the equation, these two problems can be considered unifiedly. To this end, we can unifiedly  construct the following curvature flow to solve the $L_p$-Gaussian chord Minkowski problem for $p> 0$ and log-Gaussian chord Minkowski problem for the critical $L_0$ case.

\end{remark}

Let $\partial\Omega_0$ be a smooth, closed and origin-symmetric strictly convex hypersurface in $\mathbb{R}^n$. We consider the long-time existence and convergence of the following Gauss curvature flow which is a family of convex hypersurfaces $\partial\Omega_t$  parameterized by  smooth maps $X(\cdot ,t):
S^{n-1}\times (0, \infty)\rightarrow \mathbb{R}^n$
satisfying the initial value problem
\begin{align}\label{eq109}
\left\{
\begin{array}{lc}
\frac{\partial X(x,t)}{\partial t}=-\theta(t)\frac{\langle X, v\rangle^p\mathcal{K}
(x,t)f(v)}{\widetilde{V}_{\gamma,q}(\Omega_t,\cdot)}v+X(x,t),  \\
X(x,0)=X_0(x),\\
\end{array}
\right.
\end{align}
where $\mathcal{K}(x,t)$ is the Gauss curvature of hypersurface $\partial\Omega_t$,  $v=x$ is the
outer unit normal at $X(x,t)$, $\langle X,v \rangle$ represents standard inner product of $X$ and $v$, and $\theta(t)$ is given by
\begin{align*}
\theta(t)=\frac{\int_{S^{n-1}}\widetilde{V}_{\gamma,q}(\Omega_t,\cdot)\rho^n(\xi,t)d\xi}{\int_{S^{n-1}}h^p(x,t)f(x)dx},
\end{align*}
where $\rho$ and $h$ are the radial function and support function of the convex hypersurface $\partial \Omega_t$, respectively.

\begin{remark}When $p=0$, the flow (\ref{eq109}) and $\theta(t)$ be meaningful. Thus, the flow (\ref{eq109}) can also be used to study the log-Gaussian chord Minkowski problem. 
\end{remark}

Combining problem (\ref{eq108}), (\ref{eq108+}) with flow (\ref{eq109}), we establish the following result in this article.

\begin{theorem}\label{thm13}
Suppose $q>2$, $p\geq 0$, $\partial\Omega_0$ be a smooth, closed and origin-symmetric strictly convex hypersurface in $\mathbb{R}^n$ and $f$ be a positive smooth even function on $S^{n-1}$. Then, the flow (\ref{eq109}) has a unique smooth solution to the $\partial\Omega_t=X(S^{n-1},t)$ for $t\in(0,\infty)$. When $t\rightarrow \infty$, there is a subsequence of $\partial\Omega_t$ converges in $C^\infty$to a smooth, closed, origin-symmetric and strictly convex hypersurface $\Omega_\infty$, whose support function satisfies (\ref{eq108}) for $p>0$ and (\ref{eq108+}) for $p=0$.
\end{theorem}

This paper is organized as follows. We collect background materials in Section \ref{sec2}. In Section \ref{sec3}, we give the parameterized form  of flow (\ref{eq109}) by support function and discuss properties of two important functionals along the flow (\ref{eq109}).
In Section \ref{sec4}, we give the priori estimates for the solution to the flow (\ref{eq109}). We obtain the
convergence of the flow and complete the proof of Theorem \ref{thm13} in Section \ref{sec5}.

\section{\bf Preliminaries}\label{sec2}
In this section, we give a brief review of some relevant notions about
convex bodies and recall some basic properties of convex hypersurfaces that readers may refer
to \cite{UR} and a book of Schneider \cite{SC}.
\subsection{Convex bodies} Let $\mathbb{R}^n$ be the $n$-dimensional Euclidean space, let $|z|=\sqrt{z\cdot z}$ be the Euclidean norm of $z$. The unit sphere in $\mathbb{R}^n $ is denoted by $S^{n-1}$, $\omega_n$ is the volume of the unit ball. For $k\in[0,n]$, $\mathcal{H}^k$ denotes the $k$-dimensional Hausdorff measure in $\mathbb{R}^n$. In integrals with respect to $\mathcal{H}^n$, we often abbreviate $d\mathcal{H}^n(z)$ by $dz$. Similarly, in integrals over the unite sphere $S^{n-1}$, instead of $d\mathcal{H}^{n-1}$ we write $du$. $\mathcal{H}^k_\gamma$ is the $k$-dimensional Hausdorff measure with respect to Gaussian density function $f(z)=e^{-|z^2|/2}$, $z\in\mathbb{R}^n$.

Assume that $\partial\Omega$ be a smooth, closed and strictly convex hypersurface containing the origin in its interior. The support function of convex body $\Omega$ is defined by
\begin{align*}h_\Omega(\xi)=h(\Omega,\xi)=\max\{\xi\cdot y:y\in\Omega\},\quad \forall\xi\in S^{n-1},\end{align*}
and the radial function of $\Omega$ with respect to $z\in\mathbb{R}$ is defined by
\begin{align*}\rho_{\Omega,z}(v)=\rho((\Omega,z),v)=\max\{c>0:cv+z\in\Omega\},\quad  v\in S^{n-1}.\end{align*}

For a compact convex subset $\Omega\in \mathcal{K}^n$ and $v\in S^{n-1}$, the intersection of a
supporting hyperplane with $\Omega$, $H(\Omega,v)$ at $v$ is given by
\begin{align*}H(\Omega,v)=\{y\in \Omega:y\cdot v=h_\Omega(v)\}.\end{align*}
A boundary point of $\Omega$ which only has one supporting hyperplane is called a regular point, otherwise, it is a singular point. The set of singular points is denoted as $\sigma \Omega$, it is
well known that $\sigma \Omega$ has spherical Lebesgue measure 0.

For $y\in\partial \Omega\setminus \sigma \Omega$, its Gauss map $\nu_\Omega:y\in\partial \Omega\setminus \sigma \Omega\rightarrow S^{n-1}$ is represented by
\begin{align*}\nu_\Omega(y)=\{v\in S^{n-1}:y\cdot v=h_\Omega(v)\}.\end{align*}
Correspondingly, for a Borel set $\eta\subset S^{n-1}$, its inverse Gauss map is denoted by
$\nu_\Omega^{-1}$,
\begin{align*}\nu_\Omega^{-1}(\eta)=\{y\in\partial \Omega:\nu_\Omega(y)\in\eta\}.\end{align*}
Specially, for a convex hypersurface $\partial\Omega$ of class $C^2$, then, the support function of Ω can be stated as
\begin{align*}
h(\Omega,x)=x\cdot \nu^{-1}(x)=\nu(X(x))\cdot X(x), \quad X(x)\in \partial\Omega.
\end{align*}
Moreover, the gradient of $h(\Omega, \cdot)$ satisfies
\begin{align*}
\nabla h(\Omega,x)=\nu^{-1}(x)=X(x).
\end{align*}
For the Borel set $\eta\subset S^{n-1}$, its surface area measure is defined as
\begin{align*}S_\Omega(\eta)=\mathcal{H}^{n-1}(\nu_\Omega^{-1}(\eta)).\end{align*}

\subsection{Convex hypersurface }~~Suppose that $\Omega$ is parameterized by the inverse
Gauss map $X:S^{n-1}\rightarrow \Omega$, that is $X(x)=\nu_\Omega^{-1}(x)$. Then, the support function $h$ of $\Omega$ can be computed by
\begin{align}\label{eq202}h(x)=x\cdot X(x) , \ \ x\in S^{n-1},\end{align}
where $x$ is the outer normal of $\Omega$ at $X(x)$. Let $\{e_1, e_2, \cdots, e_{n-1}\}$ be an orthonormal frame on $S^{n-1}$, denote  $e_{ij}$ by the standard metric on the sphere $S^{n-1}$.
Differentiating (\ref{eq202}), there has
\begin{align*}\nabla_ih=\nabla_ix\cdot X(x)+ x\cdot \nabla_iX(x),\end{align*}
since $\nabla_iX(x)$ is tangent to $\Omega$ at $X(x)$, thus,
\begin{align*}\nabla_ih=\nabla_ix\cdot X(x).\end{align*}

By differentiating (\ref{eq202}) twice, the second fundamental form $A_{ij}$ of $\Omega$ can be computed in terms of the support function,
\begin{align}\label{eq203}A_{ij} = \nabla_{ij}h + he_{ij},\end{align}
where $\nabla_{ij}=\nabla_i\nabla_j$ denotes the second order covariant derivative with respect to $e_{ij}$. The induced metric matrix $g_{ij}$ of $\Omega$ can be derived by Weingarten's formula,
\begin{align}\label{eq204}e_{ij}=\nabla_ix\cdot \nabla_jx= A_{ik}A_{lj}g^{kl}.\end{align}
The principal radii of curvature are the eigenvalues of the matrix $b_{ij} = A^{ik}g_{jk}$.
When considering a smooth local orthonormal frame on $S^{n-1}$, by virtue of (\ref{eq203})
and (\ref{eq204}), there has
\begin{align}\label{eq205}b_{ij} = A_{ij} = \nabla_{ij}h + h\delta_{ij}.\end{align}
Then, the Gauss curvature of $X(x)\in\Omega$ is given by
\begin{align}\label{eq206}\mathcal{K}(x) = (\det (\nabla_{ij}h + h\delta_{ij} ))^{-1}.\end{align}

\section{\bf Geometric flow and its associated functionals}\label{sec3}
In this section, we shall introduce the geometric flow and its associated functionals for solving the $L_p$-Gaussian chord Minkowski problem. For convenience, the Gauss curvature flow is restated
here. Let $\partial\Omega_0$ be a smooth, closed and origin symmetric strictly convex hypersurface in $\mathbb{R}^n$ and $f$ be a positive
smooth even function on $S^{n-1}$. We consider the following Gauss curvature flow
\begin{align}\label{eq301}
\left\{
\begin{array}{lc}
\frac{\partial X(x,t)}{\partial t}=-\theta(t)\frac{\langle X,v\rangle^p\mathcal{K}
(x,t)f(v)}{\widetilde{V}_{\gamma,q}(\Omega_t,\cdot)}v+X(x,t),  \\
X(x,0)=X_0(x),\\
\end{array}
\right.
\end{align}
where $\mathcal{K}(x,t)$ is the Gauss curvature of the hypersurface $\partial\Omega_t$ at $X(\cdot,t)$, $v=x$ is the unit outer
normal vector of $\partial\Omega_t$ at $X(\cdot,t)$, $\langle X,v\rangle$ represents standard inner product of $X$ and $v$, and $\theta(t)$ is given by
\begin{align}\label{eq302}\theta(t)=\frac{\int_{S^{n-1}}\widetilde{V}_{\gamma,q}(\Omega_t,\cdot)\rho^n(\xi,t)d\xi}{\int_{S^{n-1}}h^p(x,t)f(x)dx}.\end{align}

Taking the scalar product of both sides of the equation and of the initial condition in
(\ref{eq301}) by $v$, by means of the definition of support function (\ref{eq202}), we describe the flow equation
associated with the support function as follows
\begin{align}\label{eq303}
\left\{
\begin{array}{lc}
\frac{\partial h(x,t)}{\partial t}=-\theta(t)\frac{h^p\mathcal{K}
(x,t)f(x)}{\widetilde{V}_{\gamma,q}([h],\cdot)}+h(x,t),  \\
h(x,0)=h_0(x).\\
\end{array}
\right.\end{align}

Next, we investigate the characteristic of Guassian chord integral $I_{\gamma,q}(\Omega_t)$ along the flow (\ref{eq301}). Let's list a fact firstly (see e.g. \cite{LY}).

\begin{align}\label{eq304}\frac{1}{\rho(\xi,t)}\frac{\partial\rho(\xi,t)}{\partial t}=\frac{1}{h(x,t)}\frac{\partial
h(x,t)}{\partial t}.\end{align}

\begin{lemma}\label{lem31} For $q>1$, $p\in \mathbb{R}$, the $I_{\gamma,q}(\Omega_t)$ is  unchanged with regard to Eq. (\ref{eq303}), namely,
\begin{align*}\frac{\partial}{\partial t}I_{\gamma,q}(\Omega_t)= 0.\end{align*}
\end{lemma}

\begin{proof}
Let $h(\cdot,t)$ and $\rho(\cdot,t)$ be the support function and radial function of $\Omega_t$, respectively. From (\ref{eq101}) and (\ref{eq102}), we can derive
\begin{align}\label{eq305}
\nonumber I_{\gamma,q}(K)=&\int_K\int_K\frac{e^{-(|z|^2+|y|^2)/2}}{|z-y|^{n-q+1}}dzdy\\
\nonumber=&\int_Ke^{-|y|^2/2}\int_K\frac{e^{-|z|^2/2}}{|z-y|^{n-q+1}}dzdy\\
=&\frac{1}{2}\int_{K}\widetilde{V}_{\gamma,q}(K,y)dy.
\end{align}

Therefore, applying polar coordinates to (\ref{eq305}), by (\ref{eq302}), (\ref{eq303}), (\ref{eq304}) and $\rho^n\mathcal{K}d\xi=hdx$, we have
\begin{align*}
\frac{\partial}{\partial t}I_{\gamma,q}(\Omega_t)=&\frac{\partial}{\partial t}\bigg(\frac{1}{2}\int_{\Omega_t}\widetilde{V}_{\gamma,q}(\Omega_t,y)dy\bigg)\\
=&\frac{1}{2}\frac{\partial}{\partial t}\bigg(\int_{S^{n-1}}\int_0^{\rho(\xi,t)}\widetilde{V}_{\gamma,q}(\Omega_t,\rho(\xi,t))\rho^{n-1}d\rho d\xi\bigg)\\
=&\frac{1}{2}\int_{S^{n-1}}\widetilde{V}_{\gamma,q}(\Omega_t,\rho(\xi,t))\rho^{n-1}\frac{\partial \rho}{\partial t}d\xi\\
=&\frac{1}{2}\int_{S^{n-1}}\widetilde{V}_{\gamma,q}(\Omega_t,\rho(\xi,t))\frac{\rho^{n}\mathcal{K}}{\mathcal{K}\rho}\frac{\partial \rho}{\partial t}d\xi\\
=&\frac{1}{2}\int_{S^{n-1}}\widetilde{V}_{\gamma,q}(\Omega_t,\rho(\xi,t))\frac{h}{\mathcal{K}h}\frac{\partial h}{\partial t}dx\\
=&\frac{1}{2}\int_{S^{n-1}}\widetilde{V}_{\gamma,q}\bigg(-\theta(t)\frac{\mathcal{K}h^pf(x)}{\widetilde{V}_{\gamma,q}}+h\bigg)\frac{1}{\mathcal{K}}dx\\
=&\frac{1}{2}\bigg(-\frac{\int_{S^{n-1}}\frac{\widetilde{V}_{\gamma,q}h}{\mathcal{K}}dx}{\int_{S^{n-1}}h^pfdx}\int_{S^{n-1}}h^pf(x)dx
+\int_{S^{n-1}}\frac{\widetilde{V}_{\gamma,q}h}{\mathcal{K}}dx\bigg)\\
=&\frac{1}{2}\bigg(-\int_{S^{n-1}}\frac{\widetilde{V}_{\gamma,q}h}{\mathcal{K}}+\int_{S^{n-1}}\frac{\widetilde{V}_{\gamma,q}h}{\mathcal{K}}dx\bigg)\\
=&0.
\end{align*}
\end{proof}

For the convenience of discussing Gauss curvature flow (\ref{eq301}), we introduce a following functional for any $t\geq 0$,
\begin{align}\label{eq306}
\Phi(\Omega_t)=\frac{1}{p}\int_{S^{n-1}}f(x)h^p(x,t)dx,
\end{align}
where $h(\cdot,t)$ is the support function of $\Omega_t$ and $p\neq 0$. When $p=0$, we write
\begin{align}\label{eq307}
\Phi(\Omega_t)=\int_{S^{n-1}}\log h(x,t)f(x)dx.
\end{align}

\begin{lemma}\label{lem32}
The functional (\ref{eq306}) and (\ref{eq307}) are non-increasing along the flow (\ref{eq301}) for any $t\geq 0$. That is, $\frac{\partial}{\partial t}\Phi(\Omega_t)\leq0$.
\end{lemma}
\begin{proof}
Firstly, we prove $p\neq 0$, by (\ref{eq306}), (\ref{eq302}), (\ref{eq303}), (\ref{eq304}) and $\rho^n\mathcal{K}d\xi=hdx$, we obtain the following result,
\begin{align*}
\frac{\partial}{\partial t}\Phi_p(\Omega_t)=&\int_{S^{n-1}}h^{p-1}f(x)\frac{\partial h}{\partial t}dx\\
=&\int_{S^{n-1}}h^{p-1}f(x)\bigg(-\theta(t)\frac{\mathcal{K}h^pf(x)}{\widetilde{V}_{\gamma,q}}+h\bigg)dx\\
=&-\frac{\int_{S^{n-1}}\widetilde{V}_{\gamma,q}\frac{h}{\mathcal{K}}dx}{\int_{S^{n-1}}h^pf(x)dx}\int_{S^{n-1}}h^{2p-1}f^2(x)\mathcal{K}(\widetilde{V}_{\gamma,q})^{-1}dx+
\int_{S^{n-1}}h^pf(x)dx\\
=&\bigg(\int_{S^{n-1}}h^pf(x)dx\bigg)^{-1}\bigg\{-\int_{S^{n-1}}\widetilde{V}_{\gamma,q}\frac{h}{\mathcal{K}}dx\int_{S^{n-1}}f^2h^{2p-1}
\mathcal{K}(\widetilde{V}_{\gamma,q})^{-1}dx\\
&+\bigg(\int_{S^{n-1}}f(x)h^pdx\bigg)^2\bigg\}\\
=&\bigg(\int_{S^{n-1}}h^pf(x)dx\bigg)^{-1}\bigg\{-\bigg[\bigg(\int_{S^{n-1}}\widetilde{V}_{\gamma,q}\frac{h}{\mathcal{K}}dx\bigg)^{\frac{1}{2}}\bigg(\int_{S^{n-1}}f^2h^{2p-1}
\mathcal{K}(\widetilde{V}_{\gamma,q})^{-1}dx\bigg)^\frac{1}{2}\bigg]^2\\
&+\bigg(\int_{S^{n-1}}f(x)h^pdx\bigg)^2\bigg\}\\
=&\bigg(\int_{S^{n-1}}h^pf(x)dx\bigg)^{-1}\bigg\{-\bigg[\bigg(\int_{S^{n-1}}\bigg(\bigg(\widetilde{V}_{\gamma,q}\frac{h}{\mathcal{K}}\bigg)^\frac{1}{2}\bigg)^2dx
\bigg)^{\frac{1}{2}}\\
&\bigg(\int_{S^{n-1}}\bigg(\bigg(f^2h^{2p-1}
\mathcal{K}(\widetilde{V}_{\gamma,q})^{-1}\bigg)^\frac{1}{2}\bigg)^2dx\bigg)^\frac{1}{2}\bigg]^2+\bigg(\int_{S^{n-1}}f(x)h^pdx\bigg)^2\bigg\}\\
\leq&\bigg(\int_{S^{n-1}}h^pf(x)dx\bigg)^{-1}\bigg\{-\bigg[\int_{S^{n-1}}\bigg(\widetilde{V}_{\gamma,q}\frac{h}{\mathcal{K}}\bigg)^\frac{1}{2}\bigg(f^2h^{2p-1}
\mathcal{K}(\widetilde{V}_{\gamma,q})^{-1}\bigg)^\frac{1}{2}dx\bigg]^2\\
&+\bigg(\int_{S^{n-1}}f(x)h^pdx\bigg)^2\bigg\}\\
=&0.
\end{align*}
By the equality condition of H\"{o}lder inequality, we know that the above equality holds if and only if $\tau_1\bigg(\widetilde{V}_{\gamma,q}\frac{h}{\mathcal{K}}\bigg)^\frac{1}{2}=\bigg(f^2h^{2p-1}
\mathcal{K}(\widetilde{V}_{\gamma,q})^{-1}\bigg)^\frac{1}{2}$, i.e.,
\begin{align*}
\tau_1\widetilde{V}_{\gamma,q}h^{1-p}\det(h_{ij}+h\delta_{ij})=f.
\end{align*}

Next, we give proof of $p=0$. Taking $p=0$ in (\ref{eq302}) and (\ref{eq303}), from  (\ref{eq307}), (\ref{eq304}) and $\rho^n\mathcal{K}d\xi=hdx$, we deduce that
\begin{align*}
\frac{\partial}{\partial t}\Phi(\Omega_t)=&\int_{S^{n-1}}\frac{f(x)}{h}\frac{\partial h}{\partial t}dx\\
=&\int_{S^{n-1}}\frac{f(x)}{h}\bigg(-\theta(t)\frac{\mathcal{K}f(x)}{\widetilde{V}_{\gamma,q}}+h\bigg)dx\\
=&-\frac{\int_{S^{n-1}}\widetilde{V}_{\gamma,q}\frac{h}{\mathcal{K}}dx}{\int_{S^{n-1}}f(x)dx}\int_{S^{n-1}}\frac{f^2(x)}{h\widetilde{V}_{\gamma,q}}\mathcal{K}dx+
\int_{S^{n-1}}f(x)dx\\
=&\bigg(\int_{S^{n-1}}f(x)dx\bigg)^{-1}\bigg\{-\int_{S^{n-1}}\widetilde{V}_{\gamma,q}\frac{h}{\mathcal{K}}dx\int_{S^{n-1}}\frac{f^2}{h\widetilde{V}_{\gamma,q}}
\mathcal{K}dx+\bigg(\int_{S^{n-1}}f(x)dx\bigg)^2\bigg\}\\
=&\bigg(\int_{S^{n-1}}f(x)dx\bigg)^{-1}\bigg\{-\bigg[\bigg(\int_{S^{n-1}}\widetilde{V}_{\gamma,q}\frac{h}{\mathcal{K}}dx\bigg)^{\frac{1}{2}}\bigg(\int_{S^{n-1}}
\frac{f^2}{h\widetilde{V}_{\gamma,q}}\mathcal{K}dx\bigg)^\frac{1}{2}\bigg]^2\\
&+\bigg(\int_{S^{n-1}}f(x)dx\bigg)^2\bigg\}\\
=&\bigg(\int_{S^{n-1}}f(x)dx\bigg)^{-1}\bigg\{-\bigg[\bigg(\int_{S^{n-1}}\bigg(\bigg(\widetilde{V}_{\gamma,q}\frac{h}{\mathcal{K}}\bigg)^\frac{1}{2}\bigg)^2dx
\bigg)^{\frac{1}{2}}\bigg(\int_{S^{n-1}}\bigg(\bigg(\frac{f^2}{h\widetilde{V}_{\gamma,q}}
\mathcal{K}\bigg)^\frac{1}{2}\bigg)^2dx\bigg)^\frac{1}{2}\bigg]^2\\
&+\bigg(\int_{S^{n-1}}f(x)dx\bigg)^2\bigg\}\\
\leq&\bigg(\int_{S^{n-1}}f(x)dx\bigg)^{-1}\bigg\{-\bigg[\int_{S^{n-1}}\bigg(\widetilde{V}_{\gamma,q}\frac{h}{\mathcal{K}}\bigg)^\frac{1}{2}\bigg(\frac{f^2}
{h\widetilde{V}_{\gamma,q}}\mathcal{K}\bigg)^\frac{1}{2}dx\bigg]^2+\bigg(\int_{S^{n-1}}f(x)dx\bigg)^2\bigg\}\\
=&0.
\end{align*}
We know that the above equality holds if and only if $\tau_2\bigg(\widetilde{V}_{\gamma,q}\frac{h}{\mathcal{K}}\bigg)^\frac{1}{2}=\bigg(\frac{f^2}{h\widetilde{V}_{\gamma,q}}\mathcal{K}\bigg)^\frac{1}{2}$ by the equality condition of H\"{o}lder inequality, i.e.,
\begin{align*}
\tau_2\widetilde{V}_{\gamma,q}h\det(h_{ij}+h\delta_{ij})=f.
\end{align*}

Combining the above two situations, $\Omega_t$ satisfies (\ref{eq108}) and (\ref{eq108+}) with $\frac{1}{\tau_1}=\theta(t)(p\neq 0)$ and $\frac{1}{\tau_2}=\theta(t)(p=0)$. This completes proof of Lemma \ref{lem32}.
\end{proof}

\section{\bf Priori estimates}\label{sec4}

In this section, we establish the $C^0, C^1$ and $C^2$ estimates for the solutions to Eq. (\ref{eq303}). In the following of this paper, we always assume that $\partial\Omega_0$ is a smooth, closed and origin-symmetric strictly convex hypersurface in $\mathbb{R}^n$, $h:S^{n-1}\times [0,T)\rightarrow \mathbb{R}$ is a smooth even solution to Eq. (\ref{eq303}) with the initial $h(\cdot,0)$ the support function of $\Omega_0$. Here, $T$ is the maximal time for which the smooth solution exists to Eq. (\ref{eq303}).

\subsection{$C^0, C^1$  estimates}

In order to complete the $C^0$ estimate, we firstly need to introduce the following Lemma which was proven by Chen and Li \cite{CH} for convex bodies.
\begin{lemma}\label{lem41}\cite[Lemma 2.6]{CH}
Let $\Omega\in\mathcal{K}^n_o$, $h$ and $\rho$ be respectively support
function and radial function of $\Omega$, and $x_{\max}$ and $\xi_{\min}$ be two points such that
$h(x_{\max})=\max_{S^{n-1}}h$ and $\rho(\xi_{\min})=\min_{S^{n-1}}\rho$. Then,
\begin{align*}
\max_{S^{n-1}}h=&\max_{S^{n-1}}\rho \quad \text{and} \quad \min_{S^{n-1}}h=\min_{S^{n-1}}\rho,\end{align*}
\begin{align*}h(x)\geq& x\cdot x_{\max}h(x_{\max}),\quad \forall x\in S^{n-1},\end{align*}
\begin{align*}\rho(\xi)\xi\cdot\xi_{\min}\geq&\rho(\xi_{\min}),\quad \forall \xi\in S^{n-1}.\end{align*}
\end{lemma}

\begin{remark}
The results in Lemma \ref{lem41} be true for any $t\geq 0$, for example, we can write 
\begin{align*}h(x,t)\geq& x\cdot x^t_{\max}h(x_{\max},t),\quad \forall x\in S^{n-1}.\end{align*}
\end{remark}

\begin{lemma}\label{lem42}
Assume $p\geq 0$, $\partial\Omega_t$ be a smooth solution to the flow (\ref{eq301}) in $\mathbb{R}^n$ and $f$ is a positive smooth even function on $S^{n-1}$. Then, there is a positive constant $C$ independents of
$t$ such that
\begin{align}\label{eq401}
\frac{1}{C}\leq h(x,t)\leq C, \ \ \forall(x,t)\in S^{n-1}\times[0,T),
\end{align}
\begin{align}\label{eq402}
\frac{1}{C}\leq \rho(\xi,t)\leq C, \ \ \forall(\xi,t)\in S^{n-1}\times[0,T).
\end{align}
Here, $h(x,t)$ and $\rho(\xi,t)$ are the support function and radial function of $\Omega_t$, respectively.
\end{lemma}
\begin{proof}

We only give proof of (\ref{eq401}), and (\ref{eq402}) can be obtained by the first conclusion of Lemma \ref{lem41} and (\ref{eq401}).

Firstly, we prove the upper bound of (\ref{eq401}). From monotonicity of $\Phi(\Omega_t)$ in Lemma \ref{lem32} and the second result of Lemma \ref{lem41}, there is following result for $p>0$,
\begin{align*}
\Phi(\Omega_0)\geq&\Phi(\Omega_t)=\frac{1}{p}\int_{S^{n-1}}f(x)h(x,t)^pdx\\
\geq&\frac{1}{p}\int_{S^{n-1}}f(x)[h(x_{\max},t)x\cdot x^t_{\max}]^pdx\\
\geq&\frac{1}{p}\int_{\{x\in S^{n-1}:x\cdot x^t_{\max}\geq \frac{1}{2}\}}f(x)[h(x_{\max},t)x\cdot x^t_{\max}]^pdx\\
\geq&\frac{1}{p}\int_{\{x\in S^{n-1}:x\cdot x^t_{\max}\geq \frac{1}{2}\}}f(x)[\frac{1}{2}h(x_{\max},t)]^pdx\\
=&\frac{1}{p}\frac{h(x_{\max},t)^p}{2^p}\int_{\{x\in S^{n-1}:x\cdot x^t_{\max}\geq \frac{1}{2}\}}f(x)dx\\
\geq&Ch(x_{\max},t)^p,
\end{align*}
then,
\begin{align*}
\sup h(x_{\max},t)\leq \bigg(\frac{\Phi(\Omega_0)}{C}\bigg)^{\frac{1}{p}},
\end{align*} 
 where, $h(x_{\max},t)=\max_{S^{n-1}}h(x,t)$ for any $t$. Here, $C$ is a positive constant independents from $t$.

Similarly, when $p=0$, we have
\begin{align*}
\Phi(\Omega_0)\geq &\Phi(\Omega_t)=\int_{S^{n-1}}f(x)\log h(x,t)dx\\
\geq &\int_{S^{n-1}}f(x)\log[h(x_{\max},t)x\cdot x^t_{\max}]dx\\
\geq&\log h(x_{\max},t)\int_{S^{n-1}}f(x)dx+\int_{\{x\in S^{n-1}: x\cdot x^t_{\max}\geq\frac{1}{2}\}}f(x)\log(x\cdot x^t_{\max})\\
\geq&C\log h(x_{\max},t)-c\int_{\{x\in S^{n-1}: x\cdot x^t_{\max}\geq\frac{1}{2}\}}f(x)dx\\
\geq&C\log  h(x_{\max},t)-c_1,
\end{align*}
then, 
\begin{align*}\sup h(x_{\max},t)\leq e^{\frac{\Phi(\Omega_0)+c_1}{C}}.\end{align*}
Here, $C$ and $c_1$ be positive constants independent on $t$.

To prove the lower bound of $h(x,t)$, we use the contradiction. Let us assume that $\{t_k\}\subset [0,T)$ be a sequence such that $h(x,t_k)$ is not uniformly bounded away from $0$, i.e., $\min_{S^{n-1}}h(x,t_k)\rightarrow 0$ as $k \rightarrow \infty$. On the other hand, making use of the upper bound, by Blaschke-Selection theorem, there is a subsequence in $\{\Omega_{t_k}\}$, for convenience, which is still denoted by $\{\Omega_{t_k}\}$, such that $\{\Omega_{t_k}\}\rightarrow \widetilde{\Omega}$ as $k\rightarrow\infty$, where $\widetilde{\Omega}$ is a origin-symmetric convex body. Then, we obtain $\min_{S^{n-1}}h(\widetilde{\Omega},\cdot)=\lim_{k\rightarrow \infty}\min_{S^{n-1}}h(\Omega_{t_k},\cdot)=0$. This implies that $\widetilde{\Omega}$ is contained in a lower-dimensional subspace in $\mathbb{R}^n$. This can lead to $\rho(\xi,t_k)\rightarrow 0$ as $k\rightarrow\infty$ almost everywhere with respect to the spherical Lebesgue measure. According to bounded convergence theorem and formula (\ref{eq305}), we can derive
\begin{align*}I_{\gamma,q}(\widetilde{\Omega})=&\frac{1}{2}\int_{\widetilde{\Omega}}\widetilde{V}_{\gamma,q}(\widetilde{\Omega},y)dy\\
=&\lim_{k\rightarrow\infty}\frac{1}{2}\int_{S^{n-1}}\int_0^{\rho(\xi,t_k)}\widetilde{V}_{\gamma,q}(\widetilde{\Omega},\rho(\xi,t_k))\rho(\xi,t_k)^{n-1}d\rho d\xi \rightarrow 0.
\end{align*}
However, Lemma \ref{lem31} shows that  
\begin{align*}
I_{\gamma,q}(\widetilde{\Omega})= I_{\gamma,q}(\Omega_0)=c~~\text{(positive constant)}\neq 0,
\end{align*}
which is a contradiction. It
follows that $h(x, t)$ has a uniform lower bound. Therefore, we complete estimate of Lemma \ref{lem42}.
\end{proof}

\begin{lemma}\label{lem43}Let $p\geq 0$, $\partial\Omega_t$ be a smooth solution to the flow (\ref{eq301}) in $\mathbb{R}^n$ and $f$ is a positive smooth even function on $S^{n-1}$. Then, there is a positive constant $C$ independents of $t$ such that
\begin{align}\label{eq403}|\nabla h(x,t)|\leq C,\quad\forall(x,t)\in S^{n-1}\times [0,T),\end{align}
and
\begin{align}\label{eq404}|\nabla \rho(\xi,t)|\leq C,\quad \forall(\xi,t)\in S^{n-1}\times [0,T).\end{align}
\end{lemma}

\begin{proof}
The desired results immediately follows from Lemma \ref{lem42} and the identities (see e.g. \cite{LR}) as follows
\begin{align*}
h=\frac{\rho^2}{\sqrt{\rho^2+|\nabla\rho|^2}},\qquad\rho^2=h^2+|\nabla h|^2.\end{align*}
\end{proof}

\begin{lemma}\label{lem44}Suppose $p\geq 0$, $\partial\Omega_t$ be a smooth solution to the flow (\ref{eq301}) in $\mathbb{R}^n$ and $f$ is a positive smooth even function on $S^{n-1}$. There always exists a positive constant $C$ independents of $t$, such that
\begin{align*}\frac{1}{C}\leq\theta(t)\leq C,\quad t\in [0,T).\end{align*}
\end{lemma}

\begin{proof}By the definition of $\theta(t)$,
\begin{align*}\theta(t)=\frac{\int_{S^{n-1}}\widetilde{V}_{\gamma,q}(\Omega_t,\cdot)\rho^n(\xi,t)d\xi}{\int_{S^{n-1}}h^p(x,t)f(x)dx}.\end{align*}
Let $\Omega_t\in\mathcal{K}^n_o$, $z\in\Omega_t$, $y\in\partial\Omega_t$, $u\in S^{n-1}$, $y-z=su$, then, $dz=-s^{n-1}dsdu$, $y=\rho_{\Omega_t}(u)u$, $z=\rho_{\Omega_t}(u)u-su$ and $s\in[0, \rho_{\Omega_t,z}(u)]$. Thus,
\begin{align}\label{eq405}
\nonumber\widetilde{V}_{\gamma,q}(\Omega_t,y)=&2e^{-|y|^2/2}\int_{\Omega_t}\frac{e^{-|z|^2/2}}{|z-y|^{n-q+1}}dz\\
\nonumber=&2e^{-|\rho_{\Omega_t}(u)u|^2/2}\int_{S^{n-1}}\int_{0}^{\rho_{\Omega_t,z}(u)}\frac{e^{-|\rho_{\Omega_t}(u)u-su|^2/2}s^{n-1}}{s^{n-q+1}}dsdu\\
=&2e^{-|\rho_{\Omega_t}(u)u|^2/2}\int_{S^{n-1}}\int_{0}^{\rho_{\Omega_t,z}(u)}e^{-|\rho_{\Omega_t}(u)u-su|^2/2}s^{q-2}dsdu,
\end{align}
from the $C^0$ estimate, it suffices to have
\begin{align}\label{eq406}
\frac{1}{C}\leq\widetilde{V}_{\gamma,q}(\Omega_t,y)\leq C.
\end{align}
Therefore, the upper and lower bound of $\theta(t)$ can be directly obtained from the Lemma \ref{lem42} and (\ref{eq406}). \end{proof}

\subsection{$C^2$ estimate}

In this subsection, we establish the upper and lower bounds of principal curvature. This will shows
that Eq. (\ref{eq303}) is uniformly parabolic. The technique used in this proof was first introduced by Tso \cite{TK} to derive the upper bound of the Gauss curvature.

By Lemma \ref{lem42} and Lemma \ref{lem43}, if $h$ is a smooth even solution of Eq. (\ref{eq303}) on $S^{n-1}\times [0,T)$
and $f$ is positive smooth even function on $S^{n-1}$, then along the flow (\ref{eq301}) for $[0,T), \nabla h+hx$, and $h$ are smooth functions whose
ranges are within some bounded domain $\Omega_{[0,T)}$ and bounded interval $I_{[0,T)}$, respectively. Here $\Omega_{[0,T)}$ and $I_{[0,T)}$ depend only
on the upper and lower bounds of $h$ on $[0,T)$.

\begin{lemma}\label{lem46} For $q>2$ and $p\geq 0$, assume $\partial\Omega_t$ be a smooth solution to the flow (\ref{eq301}) in $\mathbb{R}^n$ and $f$ is a positive smooth even function on $S^{n-1}$. There is a positive constant $C$ depending on
$\|f\|_{C^0(S^{n-1})}, \|f\|_{C^1(S^{n-1})} ,\|f\|_{C^2(S^{n-1})}$, $\|h\|_{C^0(S^{n-1}\times [0,T)}$, $\|h\|_{C^1(S^{n-1}\times [0,T)}$ and $\|\theta\|_{C^0(S^{n-1}\times [0,T)}$, such that the principal curvatures $\kappa_i$ of $\Omega_t$, $i=1,\cdots, n-1$, are
bounded from above and below, satisfying
\begin{align*}\frac{1}{C}\leq \kappa_i(x,t)\leq C, \quad\forall (x,t)\in S^{n-1}\times [0,T).\end{align*}
\end{lemma}

\begin{proof} The proof is divided into two parts: in the first part, we derive an upper bound for the Gauss curvature $\mathcal{K}(x,t)$; in the second part, we give an estimate of bound above for the principal radii $b_{ij}=h_{ij}+h\delta_{ij}$.

Step 1: Prove $\mathcal{K}\leq C$.

Firstly, we construct the following auxiliary function,
\begin{align*}W(x,t)=\frac{\theta(t)\widetilde{V}_{\gamma,q}^{-1}h^pf\mathcal{K}-h}{h-\varepsilon_0}
\equiv\frac{-h_t}{h-\varepsilon_0},\end{align*}
where
\begin{align*}\varepsilon_0=\frac{1}{2}\min_{S^{n-1}\times[0, T)}h(x,t)>0,\quad h_t =\frac{\partial h}{\partial t}.\end{align*}

For any fixed $t\in[0,T)$, we assume that $W(x_0, t)=\max_{S^{n-1}}W(x,t)$ is the spatial maximum of $W$. Then at $(x_0,t)$, we have
\begin{align}\label{eq407}
0=\nabla_iW=\frac{-h_{ti}}{h-\varepsilon_0}+\frac{h_th_i}{(h-\varepsilon_0)^2},
\end{align}
and from (\ref{eq407}), at $(x_0,t)$, we also get
\begin{align}\label{eq408}
\nonumber 0\geq\nabla_{ii}W=&\frac{-h_{tii}}{h-\varepsilon_0}+\frac{h_{ti}h_{i}}{(h-\varepsilon_0)^2}
+\frac{h_{ti}h_{i}+h_th_{ii}}{(h-\varepsilon_0)^2}-\frac{h_th_i(2(h-\varepsilon_0)h_i)}{(h-\varepsilon_0)^4}\\
\nonumber=&\frac{-h_{tii}}{h-\varepsilon_0}+\frac{2h_{ti}h_{i}+h_th_{ii}}
{(h-\varepsilon_0)^2}
-\frac{2h_th_ih_i}{(h-\varepsilon_0)^3}\\
\nonumber=&\frac{-h_{tii}}{h-\varepsilon_0}+\frac{2h_{ti}h_{i}+h_th_{ii}}
{(h-\varepsilon_0)^2}
+\frac{2h_{ti}h_i}{(h-\varepsilon_0)^2}\\
=&\frac{-h_{tii}}{h-\varepsilon_0}+\frac{h_th_{ii}}
{(h-\varepsilon_0)^2}.
\end{align}
From (\ref{eq408}), we obtain
$$-h_{tii}\leq\frac{-h_th_{ii}}{h-\varepsilon_0},$$
hence,
\begin{align}\label{eq409}\nonumber -h_{tii}-h_t\delta_{ii}\leq&\frac{-h_th_{ii}}{h-\varepsilon_0}-h_t\delta_{ii}=\frac{-h_t}
{h-\varepsilon_0}(h_{ii}+(h-\varepsilon_0)\delta_{ii})\\
=&W(h_{ii}+h\delta_{ii}-\epsilon_0\delta_{ii})
=W(b_{ii}-\varepsilon_0\delta_{ii}).\end{align}
At $(x_0,t)$, we also have
\begin{align}\label{eq410}\frac{\partial}{\partial t}W=&\frac{-h_{tt}}{h-\epsilon_0}+\frac{h_t^2}{(h-\epsilon_0)^2}\\
\nonumber=&\frac{f}{(h-\epsilon_0)}\bigg[\frac{\partial(\theta(t)\widetilde{V}_{\gamma,q}^{-1}h^p)}{\partial t}\mathcal{K}+\theta(t)\widetilde{V}_{\gamma,q}^{-1}h^p\frac{\partial(\det(\nabla^2h+hI))^{-1}}{\partial t}\bigg]+W+W^2\\
\nonumber=&\frac{f}{(h-\epsilon_0)}\bigg[\bigg(\frac{\partial \theta(t)}{\partial t}\widetilde{V}_{\gamma,q}^{-1}h^p+\theta(t)\frac{\partial (\widetilde{V}_{\gamma,q}^{-1}h^p)}{\partial t}\bigg)\mathcal{K}+\theta(t)\widetilde{V}_{\gamma,q}^{-1}h^p\frac{\partial(\det(\nabla^2h+hI))^{-1}}{\partial t}\bigg]\\
\nonumber &+W+W^2.
\end{align}
Here, we firstly compute $\frac{\partial \widetilde{V}_{\gamma,q}}{\partial t}$. Recall (\ref{eq405}), we know that for $u\in S^{n-1}$,
\begin{align*}
\widetilde{V}_{\gamma,q}(\Omega_t,y)=2e^{-|\rho_{\Omega_t}(u)u|^2/2}\int_{S^{n-1}}\int_{0}^{\rho_{\Omega_t,z}(u)}e^{-|\rho_{\Omega_t}(u)u-su|^2/2}s^{q-2}dsdu.
\end{align*}
Then,
\begin{align}\label{eq411}
\nonumber\frac{\partial \widetilde{V}_{\gamma,q}}{\partial t}=&-2e^{-|\rho_{\Omega_t}(u)u|^2/2}|\rho_{\Omega_t}(u)u|\frac{\partial\rho_{\Omega_t}(u)}{\partial t}u\int_{S^{n-1}}\int_{0}^{\rho_{\Omega_t,z}(u)}e^{-|\rho_{\Omega_t}(u)u-su|^2/2}s^{q-2}dsdu\\
&+2e^{-|\rho_{\Omega_t}(u)u|^2/2}\int_{S^{n-1}}\rho_{\Omega_t,z}(u)^{q-2}e^{-|\rho_{\Omega_t}(u)u-\rho_{\Omega_t,z}(u)u|^2/2}\frac{\partial\rho_{\Omega_t,z}(u)}{\partial t}du,
\end{align}
where $\rho_{\Omega_t,z}(u)=\frac{h(\Omega_t,x)-z\cdot x}{u\cdot x}$, $z=\nabla h$, at $x_0$, there is
\begin{align*}
\frac{\partial\rho_{\Omega_t,z}(u)}{\partial t}=&\frac{\frac{\partial h}{\partial t}-\frac{\partial z}{\partial t}\cdot x}{u\cdot x}=\frac{\frac{\partial h}{\partial t}-(\nabla h)_t \cdot x}{u\cdot x}\\
=&\frac{\frac{\partial h}{\partial t}}{u\cdot x}-\frac{x}{u\cdot x}\bigg(\sum_ih_{it}e_i+h_tx\bigg)\\
=&\frac{\frac{\partial h}{\partial t}}{u\cdot x}-\frac{x}{u\cdot x}\bigg(\sum_i\frac{h_th_i}{h-\varepsilon_0}e_i+h_tx\bigg).
\end{align*}
Therefore, from Lemma \ref{lem42} and (\ref{eq411}), we can obtain
\begin{align}\label{eq412}
\nonumber\frac{\partial \widetilde{V}_{\gamma,q}}{\partial t}\leq & 2e^{-|\rho_{\Omega_t}(u)u|^2/2}\int_{S^{n-1}}\rho_{\Omega_t,z}(u)^{q-2}e^{-|\rho_{\Omega_t}(u)u-\rho_{\Omega_t,z}(u)u|^2/2}\bigg[\frac{\frac{\partial h}{\partial t}}{u\cdot x}\\
\nonumber&-\frac{x}{u\cdot x}\bigg(\sum_i\frac{h_th_i}{h-\varepsilon_0}+h_tx\bigg)\bigg]du\\
\nonumber=&2e^{-|\rho_{\Omega_t}(u)u|^2/2}\int_{S^{n-1}}\rho_{\Omega_t,z}(u)^{q-2}e^{-|\rho_{\Omega_t}(u)u-\rho_{\Omega_t,z}(u)u|^2/2}\bigg[\frac{-W(h-\epsilon_0)}{u\cdot x}\\
\nonumber&-\frac{x}{u\cdot x}\bigg(\sum_i\frac{-W(h-\epsilon_0)h_i}{h-\varepsilon_0}-W(h-\epsilon_0)x\bigg)\bigg]du\\
\nonumber\leq&2e^{-|\rho_{\Omega_t}(u)u|^2/2}\int_{S^{n-1}}\rho_{\Omega_t,z}(u)^{q-2}e^{-|\rho_{\Omega_t}(u)u-\rho_{\Omega_t,z}(u)u|^2/2}\frac{x}{u\cdot x}\bigg(\sum_i\frac{W(h-\epsilon_0)h_i}{h-\varepsilon_0}\\
\nonumber&+W(h-\epsilon_0)x\bigg)du\\
\leq &C_1W(x_0,t).
\end{align}
\begin{align*}
\nonumber\frac{\partial\theta(t)}{\partial t}=&\frac{\partial}{\partial t}\bigg(\frac{\int_{S^{n-1}}\widetilde{V}_{\gamma,q}\rho^nd\xi}{\int_{S^{n-1}}h^pf(x)dx}\bigg)\\
=&\frac{\int_{S^{n-1}}\frac{\partial \widetilde{V}_{\gamma,q}}{\partial t}\rho^nd\xi+n\int_{S^{n-1}}\widetilde{V}_{\gamma,q}\rho^{n-1}\frac{\partial \rho}{\partial t}d\xi}{\int_{S^{n-1}}h^pf(x)dx}-\frac{p\int_{S^{n-1}}\widetilde{V}_{\gamma,q}\rho^nd\xi\int_{S^{n-1}}h^{p-1}f(x)\frac{\partial h}{\partial t}dx}{\bigg(\int_{S^{n-1}}h^pf(x)dx\bigg)^2}.
\end{align*}
By $\rho^2=h^2+|\nabla h|^2$ and $(x_0,t)$ is a maximum point of $W$, we obtain
\begin{align}\label{eq413}
\frac{\partial\rho}{\partial t}=\rho^{-1}(hh_t+\sum h_kh_{kt})=\rho^{-1}W(\epsilon_0h-\rho^2)\leq\rho^{-1}W(x_0,t)(\epsilon_0h-\rho^2). 
\end{align}
From (\ref{eq401}), (\ref{eq402}), (\ref{eq406}), (\ref{eq412}) and (\ref{eq413}), we can get 
\begin{align}\label{eq414}
\frac{\partial\theta(t)}{\partial t}\leq C_2W(x_0,t).
\end{align}
For $p\geq 0$, (\ref{eq401}), (\ref{eq406}) and (\ref{eq412}) tell us that
\begin{align}\label{eq415}
\nonumber \frac{\partial(\widetilde{V}_{\gamma,q}^{-1}h^p)}{\partial t}=&-(\widetilde{V}_{\gamma,q})^{-2}\frac{\widetilde{V}_{\gamma,q}}{\partial t}h^p+(\widetilde{V}_{\gamma,q})^{-1}\frac{\partial (h^p)}{\partial t}\\
\nonumber =&-(\widetilde{V}_{\gamma,q})^{-2}\frac{\widetilde{V}_{\gamma,q}}{\partial t}h^p+(\widetilde{V}_{\gamma,q})^{-1}ph^{p-1}\frac{\partial h}{\partial t}\\
=&-(\widetilde{V}_{\gamma,q})^{-2}\frac{\widetilde{V}_{\gamma,q}}{\partial t}h^p-p(\widetilde{V}_{\gamma,q})^{-1}h^{p-1}W(h-\epsilon_0)<0.
\end{align}
We use (\ref{eq206}), (\ref{eq409}) and recall $b_{ij}=\nabla_{ij}h+h\delta_{ij}$ may give
\begin{align}\label{eq416}
\nonumber\frac{\partial (\det(\nabla^2h+hI))^{-1}}{\partial t}=&-(\det(\nabla^2h+hI))^{-2}
\frac{\partial(\det(\nabla^2h+hI))}{\partial b_{ij}}\frac{\partial(\nabla^2h+hI)}{\partial t}\\
\nonumber=&-(\det(\nabla^2h+hI))^{-2}\frac{\partial(\det(\nabla^2h+hI))}{\partial b_{ij}}
(h_{tij}+h_t\delta_{ij})\\
\nonumber\leq&(\det(\nabla^2h+hI))^{-2}\frac{\partial(\det(\nabla^2h+hI))}{\partial b_{ij}}
W(b_{ij}-\varepsilon_0\delta_{ij})\\
\leq&\mathcal{K}W((n-1)-\varepsilon_0(n-1)\mathcal{K}^{\frac{1}{n-1}}).
\end{align}
where the last inequality uses 
\begin{align*}
\sum_ib^{ii}\geq(n-1)\mathcal{K}^{\frac{1}{n-1}}.
\end{align*}
Therefore, from (\ref{eq406}), (\ref{eq414}), (\ref{eq415}) and (\ref{eq416}), we have following conclusion at $(x_0,t)$ in (\ref{eq410}),
\begin{align}\label{eq417}
\nonumber\frac{\partial}{\partial t} W\leq&\frac{f}{(h-\epsilon_0)}\bigg[\bigg(\frac{\partial \theta(t)}{\partial t}\widetilde{V}_{\gamma,q}^{-1}h^p\bigg)\mathcal{K}+\theta(t)\widetilde{V}_{\gamma,q}^{-1}h^p\frac{\partial(\det(\nabla^2h+hI))^{-1}}{\partial t}\bigg]+W+W^2\\
\leq&\frac{1}{h-\varepsilon_0}\bigg(C_3W^2+f\theta h^p\widetilde{V}_{\gamma,q}^{-1}\mathcal{K}W((n-1)-\varepsilon_0(n-1)\mathcal{K}^{\frac{1}{n-1}})\bigg)+W+W^2.
\end{align}
If $W(x,t)>> 1$($>>$: far greater than), according to construction of $W$ and the previous estimate, we easily obtain
\begin{align*}
\frac{1}{C_4}\mathcal{K}\leq W\leq C_4\mathcal{K},
\end{align*}
then, (\ref{eq417}) implies that
\begin{align*}
\frac{\partial}{\partial t}W\leq& \frac{1}{h-\varepsilon_0}\bigg(C_3W^2+f\theta h^p\widetilde{V}_{\gamma,q}^{-1}C_4W^2[(n-1)-\varepsilon_0(n-1)(C_4W) ^{\frac{1}{n-1}}]\bigg)+W+W^2\\
=&\frac{1}{h-\varepsilon_0}W^2\bigg(C_3+f\theta h^p\widetilde{V}_{\gamma,q}^{-1}C_4(n-1)-f\theta h^p\widetilde{V}_{\gamma,q}^{-1}C_4^{\frac{n}{n-1}}(n-1)\varepsilon_0W ^{\frac{1}{n-1}}+2\bigg)\\
=&\frac{f\theta h^p\widetilde{V}_{\gamma,q}^{-1}C_4^{\frac{n}{n-1}}(n-1)}{h-\varepsilon_0}W^2\bigg(\frac{C_3+f\theta h^p\widetilde{V}_{\gamma,q}^{-1}C_4(n-1)+2}{f\theta h^p\widetilde{V}_{\gamma,q}^{-1}C_4^{\frac{n}{n-1}}(n-1)}-\varepsilon_0W ^{\frac{1}{n-1}}\bigg)\\
\leq & C_5W^2(C_6-\varepsilon_0W^{\frac{1}{n-1}})<0,
\end{align*}
since $C_5$ and $C_6$ depend on $\|f\|_{C^0(S^{n-1})}$, $\|h\|_{C^0(S^{n-1}\times [0,T)}$, $\|h\|_{C^1(S^{n-1}\times [0,T)}$, $\|\theta\|_{C^0(S^{n-1}\times [0,T)}$. Consequently, we get
\begin{align*}
W(x_0,t)\leq C,
\end{align*}
and for any $(x,t)$,
\begin{align*}
\mathcal{K}(x,t)=\frac{(h-\varepsilon_0)W(x,t)+h}{\theta(t)(\widetilde{V}_{\gamma,q})^{-1}h^pf}\leq
\frac{(h-\varepsilon_0)W(x_0,t)+h}{\theta(t)(\widetilde{V}_{\gamma,q})^{-1}h^pf}\leq C.
\end{align*}

Step 2: Prove $\kappa_i\geq\frac{1}{C}$.

We consider the auxiliary function as follows
\begin{align*}
E (x,t)=\log\beta_{\max}(\{b_{ij}\})-A\log h+B|\nabla h|^2,
\end{align*}
where $A, B$ are positive constants which will be chosen later, and $\beta_{\max}(\{b_{ij}\})$ denotes the maximal eigenvalue of $\{b_{ij}\}$; for convenience, we write $\{b^{ij}\}$ for $\{b_{ij}\}^{-1}$.

For every fixed $t\in[0,T)$, suppose $\max_{S^{n-1}}E(x,t)$ is attained at point $x_0\in S^{n-1}$. By a rotation of coordinates, we may assume
\begin{align*}
\{b_{ij}(x_0,t)\} \text{ is diagonal,} \quad \text {and} \quad \beta_{\max}(\{b_{ij}\}(x_0,t))=b_{11}(x_0,t).
\end{align*}
Hence, in order to show $\kappa_i\geq\frac{1}{C}$, that is to prove $b_{11}\leq C.$ By means of the above assumption, we
transform $E(x,t)$ into the following form,
\begin{align*}
\widetilde{E}(x,t)=\log b_{11}-A\log h+B|\nabla h|^2.
\end{align*}
Utilizing again the above assumption, for any fixed $t \in [0,T)$, $\widetilde{E}(x,t)$ has a local
maximum at $x_0$, thus, we have at $x_0$,
\begin{align}\label{eq418}
0=\nabla_i\widetilde{E}=&b^{11}\nabla_ib_{11}-A\frac{h_i}{h}+2B\sum h_kh_{ki}\\
\nonumber=&b^{11}(h_{i11}+h_1\delta_{i1})-A\frac{h_i}{h}+2Bh_ih_{ii},
\end{align}
and
\begin{align*}
0\geq&\nabla_{ii}\widetilde{E}\\
=&\nabla_ib^{11}(h_{i11}+h_1\delta_{i1})+b^{11}
[\nabla_i(h_{i11}+h_1\delta_{i1})]-A\bigg(\frac{h_{ii}}{h}-\frac{h_i^2}{h^2}\bigg)
+2B(\sum h_kh_{kii}+h^2_{ii})\\
=&-(b_{11})^{-2}\nabla_ib_{11}(h_{i11}+h_1\delta_{i1})+b^{11}(\nabla_{ii}b_{11})-A\bigg(\frac{h_{ii}}{h}-\frac{h_i^2}{h^2}\bigg)
+2B(\sum h_kh_{kii}+h^2_{ii})\\
=&b^{11}\nabla_{ii}b_{11}-(b^{11})^2(\nabla_ib_{11})^2-A\bigg(\frac{h_{ii}}{h}-\frac{h_i^2}{h^2}\bigg)
+2B(\sum h_kh_{kii}+h^2_{ii}).	
\end{align*}
At $x_0$, we also have
\begin{align*}
\frac{\partial}{\partial t}\widetilde{E}=&\frac{1}{b_{11}}\frac{\partial b_{11}}{\partial t}-A\frac{h_t}{h}+2B\sum h_kh_{kt}\\ =&b^{11}\frac{\partial}{\partial t}(h_{11}+h\delta_{11})-A\frac{h_t}{h}+2B\sum h_kh_{kt}\\
=&b^{11}(h_{11t}+h_t)-A\frac{h_t}{h}+2B\sum h_kh_{kt}.
\end{align*}

From Eq. (\ref{eq303}) and (\ref{eq206}), we know that
\begin{align}\label{eq419}
\nonumber \log(h-h_t)=&\log\bigg(h+\theta h^pf\widetilde{V}_{\gamma,q}^{-1}\mathcal{K}-h\bigg)\\
\nonumber=&\log\mathcal{K}+
\log\bigg(\theta h^pf\widetilde{V}_{\gamma,q}^{-1}\bigg)\\
=&-\log[\det(\nabla^2h+hI)]+\log\bigg(\theta h^pf\widetilde{V}_{\gamma,q}^{-1}\bigg).
\end{align}
Let
\begin{align*}
\psi(x,t)=\log\bigg(\theta h^pf\widetilde{V}_{\gamma,q}^{-1}\bigg).
\end{align*}
Differentiating (\ref{eq419}) once and twice, we respectively get
\begin{align*}\frac{h_k-h_{kt}}{h-h_t}=&-\sum b^{ij}\nabla_kb_{ij}+\nabla_k\psi\\
=&-\sum b^{ii}(h_{kii}+h_i\delta_{ik})+\nabla_k\psi,
\end{align*}
and
\begin{align*}
\frac{h_{11}-h_{11t}}{h-h_t}-\frac{(h_1-h_{1t})^2}{(h-h_t)^2}=&-\bigg(-\sum (b^{ii})^2(\nabla_{i}b_{ii})^2+b^{ii}\nabla_{ii}b_{ii}\bigg)+\nabla_{11}\psi\\
=&-\sum b^{ii}\nabla_{11}b_{ii}+\sum b^{ii}b^{jj}(\nabla_1b_{ij})^2+\nabla_{11}\psi.
\end{align*}
By the Ricci identity, we have
\begin{align*}
\nabla_{11}b_{ii}=\nabla_{ii}b_{11}-b_{11}+b_{ii}.
\end{align*}
Thus, we can derive
\begin{align*}
\frac{\frac{\partial}{\partial t}\widetilde{E}}{h-h_t}=&b^{11}\bigg(\frac{h_{11t}+h_t}
{h-h_t}\bigg)-A\frac{h_t}{h(h-h_t)}
+\frac{2B\sum h_kh_{kt}}{h-h_t}\\
=&b^{11}\bigg(\frac{h_{11t}-h_{11}}{h-h_t}+\frac{h_{11}+h-h+h_t}{h-h_t}\bigg)-A\frac{1}{h}
\frac{h_t-h+h}{h-h_t}
+\frac{2B\sum h_kh_{kt}}{h-h_t}\\
=&b^{11}\bigg(-\frac{(h_1-h_{1t})^2}{(h-h_t)^2}+\sum b^{ii}
\nabla_{11}b_{ii}-\sum b^{ii}b^{jj}(\nabla_1b_{ij})^2-\nabla_{11}\psi\bigg.\\
&\bigg.+\frac{h_{11}+h-{(h-h_t)}}
{h-h_t}\bigg) -\frac{A}{h}\bigg(\frac{-(h-h_t)+h}{h-h_t}\bigg)+\frac{2B\sum h_kh_{kt}}{h-h_t}\\
=&b^{11}\bigg(-\frac{(h_1-h_{1t})^2}{(h-h_t)^2}+\sum b^{ii}
\nabla_{11}b_{ii}-\sum b^{ii}b^{jj}(\nabla_1b_{ij})^2-\nabla_{11}\psi\bigg)\\
&+b^{11}\bigg(\frac{h_{11}+h}{h-h_t}-1\bigg) +\frac{A}{h}-\frac{A}{h-h_t}+\frac{2B\sum h_kh_{kt}}{h-h_t}\\
=&b^{11}\bigg(-\frac{(h_1-h_{1t})^2}{(h-h_t)^2}+\sum b^{ii}
\nabla_{11}b_{ii}-\sum b^{ii}b^{jj}(\nabla_1b_{ij})^2-\nabla_{11}\psi\bigg)+\frac{1-A}{h-h_t}\\
&-b^{11}+\frac{A}{h}+\frac{2B\sum h_kh_{kt}}{h-h_t}\\
\leq&b^{11}\bigg(\sum b^{ii}(\nabla_{ii}b_{11}-b_{11}+b_{ii})-\sum b^{ii}b^{jj}(\nabla_1b_{ij})^2\bigg)
-b^{11}\nabla_{11}\psi+\frac{1-A}{h-h_t}\\
&+\frac{A}{h}+\frac{2B\sum h_kh_{kt}}{h-h_t}\\
\leq&\sum b^{ii}\bigg[(b^{11})^2(\nabla_ib_{11})^2+A\bigg(\frac{h_{ii}}{h}-\frac{h_i^2}{h^2}
\bigg)-2B(\sum h_kh_{kii}+h_{ii}^2)\bigg]\\
&-b^{11}\sum b^{ii}b^{jj}(\nabla_1b_{ij})^2-b^{11}\nabla_{11}\psi+\frac{1-A}{h-h_t}+\frac{A}{h}+\frac{2B\sum h_kh_{kt}}{h-h_t}\\
\leq&\sum b^{ii}\bigg[A\bigg(\frac{h_{ii}+h-h}{h}-\frac{h_i^2}{h^2}\bigg)\bigg]+2B
\sum h_k\bigg(-\sum b^{ii}h_{kii}+\frac{h_{kt}}{h-h_t}\bigg)\\
&-2B\sum b^{ii}(b_{ii}-h)^2-b^{11}\nabla_{11}\psi+\frac{1-A}{h-h_t}+\frac{A}{h}\\
\leq&\sum b^{ii}\bigg[A\bigg(\frac{b_{ii}}{h}-1\bigg)\bigg]+2B\sum h_k\bigg(\frac{h_k}{h-h_t}
+b^{kk}h_k-\nabla_k\psi\bigg)\\
&-2B\sum b^{ii}(b_{ii}^2-2b_{ii}h)-b^{11}\nabla_{11}\psi+\frac{1-A}{h-h_t}+\frac{A}{h}\\
\leq&-2B\sum h_k\nabla_k\psi-b^{11}\nabla_{11}\psi+(2B|\nabla h|-A)\sum b^{ii}-2B\sum b_{ii}\\
&+4B(n-1)h+\frac{2B|\nabla h|^2+1-A}{h-h_t}+\frac{nA}{h}.
\end{align*}
Recall
\begin{align*}
\psi(x,t)=\log\bigg(\theta h^pf(\widetilde{V}_{\gamma,q})^{-1}\bigg)=\log \theta+p\log h+\log f-\log \widetilde{V}_{\gamma,q},
\end{align*}
since $\theta$ is a constant factor, we have $\theta_k=0$.
Consequently, we may obtain following form by $\psi(x,t)$ and (\ref{eq418}),
\begin{align*}
&-2B\sum h_k\nabla_k\psi-b^{11}\nabla_{11}\psi\\
=&-2B\sum h_k\bigg(p\frac{h_k}{h}+\frac{f_k}{f}-\frac{(\widetilde{V}_{\gamma,q})_k}{\widetilde{V}_{\gamma,q}}
\bigg)-b^{11}\nabla_{11}\psi\\
=&-2B\sum h_k\bigg(p\frac{h_k}{h}+\frac{f_k}{f}-\frac{(\widetilde{V}_{\gamma,q})_k}{\widetilde{V}_{\gamma,q}}\bigg)\\
&-b^{11}\bigg(p\frac{h_{11}h-h_1^2}{h^2}+\frac{ff_{11}-f_1^2}{f^2}-\frac{(\widetilde{V}_{\gamma,q})_{11}\widetilde{V}_{\gamma,q}-(\widetilde{V}_{\gamma,q})_1^2}
{(\widetilde{V}_{\gamma,q})^2}\bigg)\\
\leq&C_7B+C_{8}b^{11}+2B\sum h_k\frac{(\widetilde{V}_{\gamma,q})_k}{\widetilde{V}_{\gamma,q}}-|p|b^{11}\frac{h(b_{11}-h)}{h^2}+b^{11}\bigg(\frac{(\widetilde{V}_{\gamma,q})_{11}\widetilde{V}_{\gamma,q}-
(\widetilde{V}_{\gamma,q})_1^2}{(\widetilde{V}_{\gamma,q})^2}\bigg)\\
=&C_7B+C_{8}b^{11}+2B\sum h_k\frac{(\widetilde{V}_{\gamma,q})_k}{\widetilde{V}_{\gamma,q}}-\frac{|p|}{h}+\frac{|p|b^{11}}{h}+b^{11}\bigg(\frac{(\widetilde{V}_{\gamma,q})_{11}\widetilde{V}_{\gamma,q}
-(\widetilde{V}_{\gamma,q})_1^2}{(\widetilde{V}_{\gamma,q})^2}\bigg)\\
\leq&C_7B+C_{9}b^{11}+2B\sum h_k\frac{(\widetilde{V}_{\gamma,q})_k}{\widetilde{V}_{\gamma,q}}+b^{11}\frac{(\widetilde{V}_{\gamma,q})_{11}
}{\widetilde{V}_{\gamma,q}}.
\end{align*}
Recall (\ref{eq405})
\begin{align*}
\widetilde{V}_{\gamma,q}(\Omega_t,y)=2e^{-|\rho_{\Omega_t}(u)u|^2/2}\int_{S^{n-1}}\int_{0}^{\rho_{\Omega_t,z}(u)}e^{-|\rho_{\Omega_t}(u)u-su|^2/2}s^{q-2}dsdu,
\end{align*}
then,
\begin{align}\label{eq420}
\nonumber (\widetilde{V}_{\gamma,q})_k=&-2e^{-|\rho_{\Omega_t}(u)u|^2/2}|\rho_{\Omega_t}(u)u|(\rho_ku+\rho\cdot e_k)\int_{S^{n-1}}\int_{0}^{\rho_{\Omega_t,z}(u)}e^{-|\rho_{\Omega_t}(u)u-su|^2/2}s^{q-2}dsdu\\
&+2e^{-|\rho_{\Omega_t}(u)u|^2/2}\int_{S^{n-1}}\rho_{\Omega_t,z}(u)^{q-2}e^{-|\rho_{\Omega_t}(u)u-\rho_{\Omega_t,z}(u)u|^2/2}(\rho_{\Omega_t,z})_kdu,
\end{align}
where $\rho_{\Omega_t,z}(u)=\frac{h(\Omega_t,x)-z\cdot x}{u\cdot x}$, $z=\nabla h$,
\begin{align}\label{eq421}
(\rho_{\Omega_t,z})_k=\bigg(\frac{h(\Omega_t,x)-\sum_ih_i\cdot x}{u\cdot x}\bigg)_k=\frac{h_k-(\sum_ih_{ik}x+\nabla h)}{u\cdot x}-\frac{(h-\nabla h\cdot x)(e_k\cdot x+u)}{(u\cdot x)^2},
\end{align}
from $\rho=(h^2+|\nabla h|^2)^{\frac{1}{2}}$, we get
\begin{align}\label{eq421+}
 \rho_k=\rho^{-1}(hh_k+\Sigma h_kh_{kk})=\rho^{-1}(hh_k+\Sigma h_k(b_{kk}-h\delta_{kk})).
\end{align}
Thus, from Lemma \ref{lem42}, (\ref{eq406}), (\ref{eq420}), (\ref{eq421}) and (\ref{eq421+}), we get
\begin{align*}
 2B\sum h_k\frac{(\widetilde{V}_{\gamma,q})_k}{\widetilde{V}_{\gamma,q}}\leq C_{10}Bb_{11}.
\end{align*}
\begin{align}\label{eq422}
\nonumber(\widetilde{V}_{\gamma,q})_{11}=&-2e^{-|\rho_{\Omega_t}(u)u|^2/2}|\rho_{\Omega_t}(u)u|(\rho_{11}u+\rho_1\cdot e_1+\rho_1\cdot e_1+\delta_{11})\times\\
\nonumber&\int_{S^{n-1}}\int_{0}^{\rho_{\Omega_t,z}(u)}e^{-|\rho_{\Omega_t}(u)u-su|^2/2}s^{q-2}dsdu\\
\nonumber&-2e^{-|\rho_{\Omega_t}(u)u|^2/2}|\rho_{\Omega_t}(u)u|(\rho_1u+\rho\cdot e_1)\int_{S^{n-1}}\rho_{\Omega_t,z}(u)^{q-2}e^{-|\rho_{\Omega_t}(u)u-\rho_{\Omega_t,z}(u)u|^2/2}(\rho_{\Omega_t,z})_kdu\\
\nonumber-&2e^{-|\rho_{\Omega_t}(u)u|^2/2}|\rho_{\Omega_t}(u)u|(\rho_1u+\rho\cdot e_1)\int_{S^{n-1}}\rho_{\Omega_t,z}(u)^{q-2}e^{-|\rho_{\Omega_t}(u)u-\rho_{\Omega_t,z}(u)u|^2/2}(\rho_{\Omega_t,z})_kdu\\
\nonumber+&2e^{-|\rho_{\Omega_t}(u)u|^2/2}\bigg[(q-2)\int_{S^{n-1}}\rho_{\Omega_t,z}(u)^{q-3}e^{-|\rho_{\Omega_t}(u)u-\rho_{\Omega_t,z}(u)u|^2/2}(\rho_{\Omega_t,z})_1^2du\\
\nonumber&+\int_{S^{n-1}}\rho_{\Omega_t,z}(u)^{q-2}e^{-|\rho_{\Omega_t}(u)u-\rho_{\Omega_t,z}(u)u|^2/2}\bigg(-|\rho_{\Omega_t}(u)u-\rho_{\Omega_t,z}(u)u|(\rho_1u+\rho e_1\\
&-(\rho_{\Omega_t,z})_1u-\rho_{\Omega_t,z}e_1)\bigg)(\rho_{\Omega_t,z})_1du+\int_{S^{n-1}}\rho_{\Omega_t,z}(u)^{q-2}e^{-|\rho_{\Omega_t}(u)u
-\rho_{\Omega_t,z}(u)u|^2/2}(\rho_{\Omega_t,z})_{11}du\bigg],
\end{align}
where 
\begin{align}\label{eq423}
\nonumber(\rho_{\Omega_t,z})_{11}=&\frac{(h_{11}-(\sum_ih_{i11}x+\sum_ih_{i1}+\sum_ih_{i1})}{u\cdot x}-\frac{(h_1-(\sum_ih_{i1}\cdot x+\nabla h)(e_1x+u)}{(u\cdot x)^2}\\
\nonumber&-\frac{(h_1-\sum_ih_{i1}x-\nabla h)(e_1\cdot x+u)+(h+(\nabla h\cdot x))(\delta_{11}\cdot x+2e_1)}{(u\cdot x)^2}\\
&+\frac{2(e_1\cdot x+u)(h-\nabla h\cdot x)(e_1\cdot x+u)}{(u\cdot x)^3}.
\end{align}
From (\ref{eq421+}), we get
\begin{align*}
\rho_{11}=\frac{hh_{11}+h_1^2+\Sigma h_1h_{111}+\Sigma h_{11}^2}{\rho}-\frac{h_1^2b_{11}^2}{\rho^3}.
\end{align*}
This combined with (\ref{eq418}) implies
\begin{align}\label{eq424}
\rho_{11}=\frac{h(b_{11}-h)+h_1^2+\Sigma h_1\bigg(A\frac{h_1}{h}-2Bh(b_{11}-h)-b^{11}(h_1\delta_{11})\bigg )b_{11}}{\rho}-\frac{h_1^2b_{11}^2}{\rho^3}.
\end{align}

By Lemma \ref{lem42}, Lemma \ref{lem43}, (\ref{eq406}), (\ref{eq422}), (\ref{eq418}), (\ref{eq423}) and (\ref{eq424}), we conclude for $q>2$, 
\begin{align*}
b^{11}\frac{(\widetilde{V}_{\gamma,q})_{11}}{\widetilde{V}_{\gamma,q}}\leq C_{11}b_{11}.
\end{align*}
It follows that
\begin{align*}
\frac{\frac{\partial}{\partial t}\widetilde{E}}{h-h_t}\leq C_{12}Bb_{11}+C_{13}b^{11}+
(2B|\nabla h|-A)\sum b^{ii}-2B\sum b_{ii}+4B(n-1)h+\frac{nA}{h}<0,
\end{align*}
provided $b_{11}>>1$ and if we choose $A>>B$. We obtain
\begin{align*}
\widetilde{E}(x_0,t)\leq C,
\end{align*}
hence,
\begin{align*}
E(x_0,t)=\widetilde{E}(x_0,t)\leq C.
\end{align*}
This tells us the principal radii are bounded from above, or equivalently $\kappa_i\geq\frac{1}{C}$.
\end{proof}

\vskip 0pt
\section{\bf The convergence of the flow}\label{sec5}

With the help of priori estimates in the section \ref{sec4}, the long-time existence and asymptotic behaviour of the flow (\ref{eq109}) (or (\ref{eq301})) are obtained, we also can complete proof of Theorem \ref{thm13}.
\begin{proof}[Proof of the Theorem \ref{thm13}] Since Eq. (\ref{eq303}) is parabolic, we can get its short time existence. Let $T$ be the maximal time such that $h(\cdot, t)$ is a smooth even solution
to Eq. (\ref{eq303}) for $t\in[0,T)$. Lemma \ref{lem42}-\ref{lem44} enable us to apply Lemma \ref{lem46} to Eq. (\ref{eq303}), thus, we can deduce a
uniformly upper and lower bounds for the biggest eigenvalue of $\{(h_{ij}+h\delta_{ij})(x,t)\}$. This implies
\begin{align*}
	C^{-1}I\leq (h_{ij}+h\delta_{ij})(x,t)\leq CI,\quad \forall (x,t)\in S^{n-1}\times [0,T),
\end{align*}
where $C>0$ independents on $t$. This shows that Eq. (\ref{eq303}) is uniformly parabolic. Estimates for higher derivatives follows from the standard regularity theory of uniformly parabolic equations
Krylov \cite{KR}. Hence, we obtain the long time existence and regularity of solutions for the flow
(\ref{eq109}) (or (\ref{eq301})). Moreover, we obtain
\begin{align}\label{eq501}
\|h\|_{C^{l,m}_{x,t}(S^{n-1}\times [0,T))}\leq C_{l,m},
\end{align}
for some $C_{l,m}$ ($l, m$ are nonnegative integers pairs) independent of $t$, then $T=\infty$. Using parabolic comparison principle, we can attain the uniqueness of the smooth even solution $h(\cdot,t)$ of Eq. (\ref{eq303}).

Now, recall the non-increasing property of $\Phi(\Omega_t)$ in Lemma \ref{lem32}, we know that for $p\geq 0$,

\begin{align}\label{eq502}
\frac{\partial\Phi(\Omega_t)}{\partial t}\leq 0.
\end{align}
Based on (\ref{eq502}), there exists a $t_0$ such that
\begin{align*}
\frac{\partial\Phi(\Omega_t)}{\partial t}\bigg|_{t=t_0}=0,
\end{align*}
this yields 
\begin{align*}
\tau_1\widetilde{V}_{\gamma,q}(\Omega_{t_0},\cdot)h_{\Omega_{t_0}}^{1-p}\det(h_{ij}+h\delta_{ij})=f, \quad p>0.
\end{align*}
and
\begin{align*}
\tau_2\widetilde{V}_{\gamma,q}(\Omega_{t_0},\cdot)h_{\Omega_{t_0}}\det(h_{ij}+h\delta_{ij})=f,  \quad p=0.
\end{align*}
Let $\Omega=\Omega_{t_0}$, thus, $\Omega$ satisfies (\ref{eq108}) for $p>0$ and (\ref{eq108+}) for $p=0$.

In view of (\ref{eq501}), applying the Arzel$\grave{a}$-Ascoli theorem and a diagonal argument, we can extract
a subsequence of $t$, denoted by $\{t_j\}_{j \in \mathbb{N}}\subset (0,+\infty)$, and there exists a smooth even function $h(x)$ such that
\begin{align}\label{eq503}
\|h(x,t_j)-h(x)\|_{C^i(S^{n-1})}\rightarrow 0,
\end{align}
uniformly for each nonnegative integer $i$ as $t_j \rightarrow \infty$. This reveals that $h(x)$ is a support function. Let us denote by $\Omega$ the convex body determined by $h(x)$. Thus, $\Omega$ is smooth, origin-symmetric and strictly convex.

Moreover, by (\ref{eq501}) and the uniform estimates in section \ref{sec4}, we conclude that $\Phi(\Omega_t)$ is a bounded function in $t$ and $\frac{\partial \Phi(\Omega_t)}{\partial t}$ is uniformly continuous. Thus, for any $t>0$, by monotonicity of the $\Phi(\Omega_t)$ in Lemma \ref{lem32}, there is a constant $C>0$ independents of $t$, such that
\begin{align*}
\int_0^t\bigg(-\frac{\partial\Phi(\Omega_t)}{\partial t}\bigg)dt=\Phi(\Omega_0)-\Phi(\Omega_t)\leq C,
\end{align*}
this gives
\begin{align}\label{eq504}\lim_{t\rightarrow\infty}\Phi(\Omega_t)-\Phi(\Omega_0)=
-\int_0^\infty\bigg|\frac{\partial}{\partial t}\Phi(\Omega_t)\bigg|dt\leq C.
\end{align}
The left hand side of (\ref{eq504}) is bounded below by $-2C$, therefore, there is a subsequence $t_j\rightarrow\infty$ such that
\begin{align*}
\frac{\partial}{\partial t}\Phi(\Omega_{t_j})\rightarrow 0 \quad\text{as}\quad  t_j\rightarrow\infty.
\end{align*}
The proof of Lemma \ref{lem32} shows that when $p>0$,
\begin{align}\label{eq505}
\bigg(\int_{S^{n-1}}f(x)(h^\infty)^pdx\bigg)^2=\int_{S^{n-1}}\widetilde{V}_{\gamma,q}\frac{h^\infty}{{\mathcal{K}}}dx
\int_{S^{n-1}}f^2(h^\infty)^{2p-1}{\mathcal{K}}\widetilde{V}_{\gamma,q}^{-1}dx,
\end{align}
and $p=0$,
\begin{align}\label{eq506}
\bigg(\int_{S^{n-1}}f(x)dx\bigg)^2=\int_{S^{n-1}}\widetilde{V}_{\gamma,q}\frac{h^\infty}{{\mathcal{K}}}dx
\int_{S^{n-1}}\frac{f^2{\mathcal{K}}}{(h^\infty)\widetilde{V}_{\gamma,q}}dx,
\end{align}
where $h^{\infty}$ is the support function of $\Omega^{\infty}$. By equality condition of H\"{o}lder inequality, (\ref{eq505}) and (\ref{eq506}) means that for $p>0$,
\begin{align}\label{eq507}
\tau_1\widetilde{V}_{\gamma,q}(\Omega,\cdot)(h^\infty_{\Omega})^{1-p}\det(h^\infty_{ij}+h^\infty\delta_{ij})=f(x),
\end{align}
and $p=0$,
\begin{align}\label{eq508}
\tau_2\widetilde{V}_{\gamma,q}(\Omega,\cdot)h^\infty_{\Omega}\det(h^\infty_{ij}+h^\infty\delta_{ij})=f(x),
\end{align}
(\ref{eq507}) satisfies (\ref{eq108}) and (\ref{eq508}) satisfies (\ref{eq108+})
with $\tau_1$ and $\tau_2$ given by
\begin{align*}\frac{1}{\tau_1}=\lim_{t_j\rightarrow\infty}\theta(t_j),\quad p>0,
\end{align*}
and 
\begin{align*}\frac{1}{\tau_2}=\lim_{t_j\rightarrow\infty}\theta(t_j),\quad p=0,
\end{align*}
This completes the proof of Theorem \ref{thm13}.
\end{proof}

\end{document}